\documentclass[a4paper,12pt]{amsart}
\usepackage{amsmath}

\usepackage[utf8]{inputenc}

\newtheorem{theorem}{Theorem}[section]
\newtheorem{proposition}[theorem]{Proposition}

\theoremstyle{definition}
\newtheorem{definition}[theorem]{Definition}

\theoremstyle{remark}
\newtheorem{remark}[theorem]{Remark}

\newcommand{\hpart}{h^{1,q}}
\newcommand{\Fpart}{F^{1,q}}
\newcommand{\Zpart}{Z^{1,q}}
\newcommand{\Spart}{S^{1,q}}

\newcommand{\hspin}{h^{r,1}}
\newcommand{\Fspin}{F^{r,1}}
\newcommand{\Zspin}{Z^{r,1}}
\newcommand{\Sspin}{S^{r,1}}

\newcommand{\hmix}{h^{r,q}}

\newcommand{\Zmix}{Z^{r,q}}
\newcommand{\Smix}{S^{r,q}}

\newcommand{\la}{\left\langle}
\newcommand{\ra}{\right\rangle}
\newcommand{\E}{\mathcal{E}}

\newcommand{\C}{\mathbb{C}}
\newcommand{\xdiff}{\frac{\mathrm{d}}{\mathrm{d}x}}
 
\newcommand{\Z}{\mathbb{Z}}
\renewcommand{\la}{\left\langle}
\renewcommand{\ra}{\right\rangle}

\title[Schr\" odinger equation of Hurwitz numbers]{The spectral curve and the Schr\" odinger equation of double Hurwitz numbers and higher spin structures}

\author{M.~Mulase}
\thanks{M.~M. is supported by the 
NSF grant DMS-1104734. S.~S. and L.~S. are suppoted by a VIDI grant of NWO}

\address{M.M.: Department of Mathematics, University of California, Davis, CA 95616–8633, U.S.A.}

\email{mulase@math.ucdavis.edu}

\author{S.~Shadrin}

\author{L.~Spitz}

\address{S.S. and L.S.: Korteweg-de Vries Institute for Mathematics, University of Amsterdam, Postbus 94248, 1090 GE Amsterdam, The Netherlands}

\email{s.shadrin@uva.nl, l.spitz@uva.nl}

\begin{document}

\begin{abstract} We derive the spectral curves for $q$-part double Hurwitz numbers, $r$-spin simple Hurwitz numbers, and  arbitrary combinations 
of these cases, from the analysis of the 
unstable $(0,1)$-geometry. We quantize this family of spectral curves and obtain the Schr\"odinger equations for the partition function of the corresponding Hurwitz problems. We thus confirm the conjecture for the existence of
\emph{quantum curves} 
in these generalized Hurwitz number cases.
\end{abstract}

\maketitle


\section{Introduction and the main results}

The purpose of this paper is to 
rigorously solve the conjecture
of  
\cite{ADKMV, DHS, DHSV, DV2007,GuSu}
 for the existence
of the \emph{quantum curves} 
for three series of 
infinitely many cases of generalized Hurwitz  numbers. 
The semi-classical limit of these quantum curves
recovers the spectral curves of the Eynard-Orantin
integral recursion for each of these cases.
Our main results, the concrete 
formulas for the spectral curves
and their quantization, are presented in Table 1 and Table 2
below.

\subsection{Hurwitz numbers and Eynard-Orantin recursion}
Simple
Hurwitz numbers $h_{g,\mu}$ enumerate 
genus $g$ ramified covering of $\mathbb{P}^1$, with one special fiber over infinity, where the cyclic type of the monodromy is given by the sequence $\mu=(\mu_1,\dots,\mu_\ell)$, and with $m:=2g-2+\ell+\sum_{i=1}^\ell \mu_i$ simple critical points. 

Hurwitz numbers play an important role in 
various areas of mathematics, such 
as  combinatorics and representation theory 
of  symmetric groups, integrable systems, and Hodge integrals over  the moduli spaces of curves. 
One of the recent exciting developments
 on Hurwitz 
numbers is the discovery of their relation to  random matrix theory and 
related fields, in particular, to the Eynard-Orantin   integral recursion formalism. 

The Eynard-Orantin recursion~\cite{EO} is 
an effective algorithm to  calculate various quantum 
 invariants, such as closed and
open Gromov-Witten invariants of toric target
spaces and certain Hurwitz numbers.
The formula calculates these invariants 
from rather small input data, that consist of
only a plane algebraic or analytic curve, 
called the \emph{spectral curve} of the
theory, and 
 Riemann's normalized fundamental 
differential form of the  second kind defined on 
the spectral curve.

We learn from  various physics literature 
\cite{ADKMV, DHS, DHSV, DV2007,GuSu}
that when the spectral curve has genus $0$,  
it is conjectured that the following holds.
\begin{itemize}

\item There exists a unique procedure to calculate
the canonical primitive functions of the 
symmetric differential forms that are 
obtained by the Eynard-Orantin 
integral recursion.

\item The  \emph{partition function} of the theory, 
which is the exponential generating function
of the \emph{principal specialization} of 
these primitive functions,  
satisfies a holonomic system generated by 
a single stationary Schr\"odinger operator.

\item Moreover, the total symbol of the holonomic
system defines a Lagrangian subvariety
immersed into 
the cotangent bundle of $\mathbb{C}^*$, which is 
exactly the same as the realization of the 
spectral curve as a plane curve.

\item In other words, the spectral curve and
its immersion 
as a Lagrangian into the cotangent bundle are recovered
from the semi-classical limit of the 
Schr\"odinger equation.
\end{itemize}

In the physics literature cited above, this
Schr\"odinger operator is called 
a \emph{quantum curve}. It is the
Weyl quantization of the 
defining equation of the spectral curve
in the cotangent bungle.
A mathematical proof of  
this conjecture for a few simple cases have 
been established in \cite{MS}.

The generating functions of simple 
Hurwitz numbers satisfy the Ey\-nard-Orantin
integral recursion with the Lambert curve
$x=ye^{-y}$ as the spectral curve in the
$xy$-plane. It was originally conjectured 
by Bouchard and Mari\~no \cite{BM}, and 
mathematically proved in many different ways in 
\cite{BEMS,EMS09,Eyn}. In~\cite{Zhou2} Zhou showed the existence of the quantum
curve for the 
case of simple Hurwitz numbers, quantizing in a proper way the equation of the Lambert curve (see also~\cite{MS}).

\subsection{Generalizations of Hurwitz numbers}

In this paper we consider two different generalizations of the usual simple Hurwitz numbers. One of them is the double Hurwitz numbers that count the ramified coverings of $\mathbb{P}^1$ with two special fibers. One of which has an arbitrary fixed cyclic type of monodromy $\mu=(\mu_1,\dots,\mu_\ell)$, and the other  has the cyclic type of monodromy equal to $(q,q,\dots,q)$. All other critical points are 
assumed to be simple. This
type of Hurwitz numbers we call 
\emph{$q$-double Hurwitz numbers}
 and denote by $\hpart_{g,\mu}$. There is a closed formula for these numbers in terms of the so-called Hurwitz-Hodge integrals, see~\cite{JPT}. For $q=1$ we recover the usual 
simple Hurwitz numbers.

Another generalization of simple 
Hurwitz numbers is the 
so-called \emph{$r$-spin Hurwitz numbers},
 denoted by $\hspin_{g,\mu}$. In this case we can intuitively think that we have  an arbitrary fixed cyclic type of monodromy $\mu$ at
 a special fiber, but all other ramifications are \emph{completed} $(r+1)$-cycles, instead of the usual simple critical points. These 
 completed cycles can be naturally defined as certain special elements of the center of the group algebra of the symmetric group~\cite{KerOls}.  They also play
a  key role in the Gromov-Witten theory of 
$\mathbb{P}^1$  \cite{OkoPan06}. There is a closed formula for these numbers in terms of the intersection theory of the moduli space of $r$-spin structures conjectured by Zvonkine~\cite{Zvo} and proved in~\cite{ShaSpiZvo2}. The geometric and algebraic definitions of these numbers are discussed in detail in~\cite{ShaSpiZvo}, though in this paper we use a slightly different normalization.

Finally we shall consider the mixed case of the
above two generalizations. Geometrically, this 
is the case of two special fibers, where one has an arbitrary fixed monodromy, the other
 has the cyclic type of  $(q,q,\dots,q)$, and all other ramifications are the completed $(r+1)$-cycles. We call these numbers 
\emph{$q$-double $r$-spin Hurwitz numbers},
and denote them by $\hmix_{g,\mu}$.

\subsection{Spectral curves}

If we have a  partition function $Z$ that is the
exponential generating function
of the \emph{free energies}, i.e., if $Z$ has an
expansion of the form
 $Z=\exp\left(\sum_{g=0}^\infty \lambda^{2g-2}\sum_{\ell=1}^\infty F_{g,\ell}\right)$, 
 then a natural question is whether we can produce a spectral curve and the other input data of the Eynard-Orantin recursion procedure so that the  $\ell$-point differential 
 forms $\omega_{g,\ell}$ determined by the recursion would coincide with the
 exterior derivatives $d_1\cdots d_\ell F_{g,\ell}$
 of the free energies.

We do not have a general answer to this question.
If we can find the holonomic system satisfied 
by $Z$, then its semi-classical limit gives 
a spectral curve as a holomorphic 
Lagrangian subvariety. Another 
mechanism
was proposed in~\cite{DMSS12}.
The idea is that the spectral curve can be obtained via the analysis of the $(0,1)$-geometry, that is, 
the spectral curve is the Riemann surface (the maximal
domain of holomorphy) of 
the one variable function $F_{0,1}$. This mechanim
works for many examples, including simple
Hurwitz numbers \cite{DMSS12}.

Note that in both cases, it is not \`a priori clear that the $\ell$-point differential forms produced from the resulting spectral curve will coincide with the exterior derivatives of the free energies; it appears to be the case in many known examples, but has to be proved in each individual case.

We first 
examine the latter idea in the case of 
various generalizations of simple 
Hurwitz numbers described above. This way we obtain the spectral curves in Table~\ref{tab:spectral_curves}.

\begin{table}[h!tb] 
  \centering
  
  \begin{tabular}{|c||c|}

\hline 

$q$-Double Hurwitz Numbers
 & $
x=y^{1/q} e^{-y} 
$
\tabularnewline
\hline 
$r$-Spin Hurwitz Numbers & $
x=ye^{-y^r}
$  
\tabularnewline
\hline 
Mixed $q$-Double $r$-Spin Hurwitz Numbers& 
$x=y^{1/q}e^{-y^r}$  \tabularnewline
\hline 

\end{tabular}
\bigskip

  \caption{Spectral Curves.}\label{tab:spectral_curves}
\end{table}

We note that the
spectral curve for the case of $q$-double Hurwitz numbers was recently proved in~\cite{BHM13,DLN}.
Evidence for the formula for the spectral curve for $r$-spin Hurwitz numbers is given in~\cite{ShaSpiZvo2}. The mixed case is so far still conjectural.

\subsection{Schr\"odinger equations} The
formulas for the spectral curves,
even still conjectural for the most general 
case, give enough input to test the conjecture of 
the existence of the quantum curves, 
or the Schr\"odinger equation for the principal specialization of the 
partition function. We prove it in all three cases mentioned above, generalizing in this way the result of \cite{Zhou2} for  simple Hurwitz numbers.

It is worth mentioning that when we apply Weyl quantization, we need to find the 
correct ordering of the operators. Our 
guiding principle is the straightforward application of
 the semi-infinite wedge product formalism
 of the various Hurwitz numbers. 
 
 The main result of the quantum curves 
 we establish are summarized
 in the following table.
 
 \begin{table}[htb]
  \centering
  
  \begin{tabular}{|c||c|}

\hline 

$q$-Double Hurwitz Numbers &
$
\hat y - \left(e^{\frac{q-1}{2} \hat y} \hat x e^{-\frac{q-1}{2} \hat y} \right)^q e^{q \hat y}
$
\tabularnewline
\hline 
$r$-Spin Hurwitz Numbers & $
\hat y - \hat x^{\frac{3}{2}} \exp\left(\frac{\sum_{i=0}^r \hat x^{-1} \hat y^i \hat x \hat y^{r-i}}{r+1}\right) \hat x^{-\frac{1}{2}}
$  
\tabularnewline
\hline 
Mixed  Hurwitz Numbers& 
$\hat y-\hat x^{q+1/2} e^{\frac{q}{r+1}\sum_{i=0}^r \hat x^{-q} \hat y ^i \hat x^{q} \hat y^{r-1}} \hat x^{-1/2}$  \tabularnewline
\hline 

\end{tabular}
\bigskip

  \caption{Quantum Curves.}
\end{table}

Here the canonical 
quantization of the coordinate
functions $x$ and $y$ are defined by 
$$
\begin{cases}
\hat x = x
\\
\hat y = \lambda x\frac{d}{dx},
\end{cases}
$$
reflecting the nature of the cotangent bundle
$T^*(\mathbb{C}^*)$ and the holomorphic
tautological $1$-form $y d\log x$ on it.

\subsection{Organization of the paper}

In Section 2 we collect the necessary
 background materials of the semi-infinite wedge formalism. After this preparation, in each of the
 following three sections we  
 (a) define a particular generalization of 
 Hurwitz numbers; (b) derive the formula for the principal specialization of their partition function;  
 (c) identify the formula for the spectral curve; and 
 (d) prove the existence of the 
 quantum curve, or the stationary Schr\"odinger equation.
 The $q$-double Hurwitz numbers are studied in
 Section 3,  the $r$-spin Hurwitz numbers
 in Section 4, and  finally in Section 5 we prove the
 results for the mixed case.

\section{Infinite-wedge space}
In this section we sketch the theory of the semi-infinite wedge space. We will use it to express both the $q$-double Hurwitz numbers and the $r$-spin Hurwitz numbers (in fact, in this paper we use these expressions as definitions) and to compute the corresponding spectral curves and their quantizations. Here we give just a quick reminder of these things; we refer to
\cite{Joh10, OkoPan06, ShaSpiZvo}
for more detail.

The infinite wedge space is defined in the following way. Let $V$ be an infinite dimensional vector space with basis labelled by the half integers. Denote by $\underline{i}$ the basis vector labelled by $i$, so $V = \bigoplus_{i \in \Z + \frac{1}{2}} \underline{i}$.

\begin{definition}\label{Def:WedgeProduct} Let $c$ be an integer.
An \emph{infinite wedge product of charge $c$} is a formal expression 
\begin{equation}\label{eq:wedgeProduct}
\underline{i_1} \wedge \underline{i_2} \wedge \cdots
\end{equation}
such that the sequence of half-integers $i_1, i_2, i_3, \dots$ differs from the sequence $c-1/2, c-3/2, c-5/2, \dots$ in only a finite number of places.

The \emph{charged infinite wedge space} is the span of all infinite wedge products. The \emph{infinite wedge space} is its zero charge subspace, that is, the span of all zero charge infinite wedge products.
On both spaces we introduce the inner product $(\cdot,\cdot)$ for which the vectors of the form~(\ref{eq:wedgeProduct}) are orthonormal.
\end{definition}

Note that the infinite wedge space is spanned by the vectors
\begin{equation}
v_\lambda = \underline{\lambda_1 - 1/2} \wedge \underline{\lambda_2 - 3/2} \wedge \underline{\lambda_3 - 5/2} \wedge \dots,
\end{equation}
where $(\lambda_1 \geq \lambda_2 \geq \dots \geq 0 \geq 0 \geq \dots)$ is a partition of any non-negative integer.

The Hurwitz numbers will be expressed as so-called \emph{vacuum expectation values} of some appropriate operators.  

\begin{definition}
The zero charge vector $v_\emptyset = \underline{-\frac{1}{2}} \wedge \underline{-\frac{3}{2}} \wedge \cdots$ is denoted by $|0 \rangle$ and is called the {\em vacuum vector}. Its dual $\langle 0 |$ with respect to the inner product is called the {\em covacuum} vector. If $\mathcal{P}$ is an operator on the infinite wedge space, we define its  \emph{vacuum expectation value} as
$\langle \mathcal{P}\rangle
= \langle 0 |\mathcal{P}|0\rangle$.
\end{definition}

\begin{definition} Let $k$ be any half integer. Then the operator $\Psi_k$ is defined by
$\Psi_k \colon (\underline{i_1} \wedge \underline{i_2} \wedge \cdots) \ \mapsto \ (\underline{k} \wedge \underline{i_1} \wedge \underline{i_2} \wedge \cdots)$. This operator acts on the charged infinite wedge space and increases the charge by $1$.

The operator $\Psi_k^*$ is defined to be the adjoint of the operator $\Psi_k$ with respect to the inner product.

The normally ordered products of $\Psi$-operators are defined in the following way
\begin{equation}
{:}\Psi_i \Psi_j^*{:} \ = \begin{cases}\Psi_i \Psi_j^*, & \text{ if } j > 0 \\
-\Psi_j^* \Psi_i & \text{ if } j < 0\ .\end{cases} 
\end{equation}
\end{definition}

Note that the two expressions are equal unless $i=j$. Also note that the operator~${:}\Psi_i \Psi_j^*{:}$ does not change the charge of an infinite wedge product, and can thus be viewed as an operator on the infinite wedge space.

\begin{definition}
Let $n \in \Z$ be any integer. We define the so-called \emph{$\E$-operators} $\E_n(z)$ and~$\tilde{\E}_n(z)$ depending on a formal variable~$z$ by
\begin{align}
\E_n(z) &= \sum_{k \in \Z + \frac{1}{2}} e^{z(k - \frac{n}{2})} \, {:} \Psi_{k-n} \Psi^*_k {:} \, + \frac{\delta_{n,0}}{e^{z/2} - e^{-z/2}} \\
\tilde{\E}_n(z) &= \sum_{k \in \Z + \frac{1}{2}} e^{z(k - \frac{n}{2})} \, {:} \Psi_{k-n} \Psi^*_k {:} \, .
\end{align}
If $n \neq 0$ we denote by~$\alpha_n$ the operator
\begin{equation}
\alpha_n = \sum_{k \in \Z + \frac{1}{2}} {:} \Psi_{k-n} \Psi^*_k {:} \; .
\end{equation}
\end{definition}

Informally, the operator $\alpha_n$ attempts to add $n$ to every factor of an infinite wedge product and returns the sum of successful attempts.

The vacuum expectation values of a product $\E_{a_1}(z_1) \cdots \E_{a_n}(z_n)$ of these operators with $\sum a_i = 0$ are computed using the following facts. 

\begin{proposition}\label{Prop:Ecommut}
Denote by~$\zeta$ the function 
\begin{equation}
\zeta(z) := e^{z/2} - e^{-z/2} \ . 
\end{equation} 
Then we have
\begin{equation}\label{eq:comm}
[\E_k(w), \E_l(z)] = \zeta(kz-lw) \E_{k+l}(z+w);
\end{equation}
in particular,
\begin{equation}
[\E_k(0) , \E_l(z)] = \zeta(kz) \E_{k+l}(z)
\end{equation}
and, taking a limit as $z \to 0$,
\begin{equation}
[\E_k(0) , \E_l(0)] = k  \delta_{k+l,0}.
\end{equation}
Note that the proposition is still true when we replace any of the $\E$-operators on the left-hand side by the corresponding~$\tilde{\E}$. 
\end{proposition}

\begin{proposition}\label{Prop:zero}
We have $\E_n(z)|0\rangle = 0 $ for $n>0$, and $\langle 0|\E_n(z) = 0$ for $n<0$. We also have $\tilde{\E}_0(z)|0\rangle = 0$ and $\langle 0|\tilde{\E}_0(z) = 0$ . 
\end{proposition}

The vacuum expectation value of a product of $\E$-operators is now computed by commuting the operators~$\E_n(z)$ with negative~$n$ to the left using Proposition~\ref{Prop:Ecommut}. Repeating this will lead to either a zero contribution by Proposition~\ref{Prop:zero}, or some product of operators of the form~$\E_0(z)$. The last can be calculated using 
\begin{equation}
\E_0(z) |0\rangle = \frac{1}{\zeta(z)} |0\rangle .
\end{equation}

At the end of the calculation, there can be either one or more $\E$-operators in the vacuum expectation value. If there is one, this is called a \emph{connected contribution} to the vacuum expectation value, otherwise it is a \emph{disconnected contribution}. It turns out that the sum of connected contributions in well-defined (does not depend on the order in which we have computed the commutators); it is called the \emph{connected vacuum expectation value} and denoted by~$\langle \cdot \rangle^\circ$ (adding a super-script circle to the full vacuum expectation value).

\section{$q$-Double Hurwitz numbers} \label{sec:rpart}
In this section, we study \emph{$q$-double Hurwitz numbers}. Their geometric definition, mentioned in the Introduction, is equivalent~(\cite{Oko,Joh10}) to the following one in terms of vacuum expectation values in the infinite wedge space.

\begin{definition} We define the (connected) $q$-double Hurwitz numbers as
\begin{equation}
\hpart_{g;\mu} := [w_1^2 \cdots w_m^2] \la \prod_{i=1}^{l(\mu)} \frac{\alpha_{\mu_i}}{\mu_i} \cdot \prod_{j=1}^m \tilde{\E}_0(w_j) \cdot \frac{(\alpha_{-q})^s}{q^s\cdot s!} \ra^\circ ,
\end{equation}
where $[w_1^{d_1} \cdots w_n^{d_n}]$ denotes the coefficient of the monomial~$w_1^{d_1} \cdots w_n^{d_n}$ in the power series that follows it. Note that~$s = |\mu|/q$ is an integer since~$|\mu|$ is the degree of the covering, and~$m$ is the number of simple ramification points away from~$0$ and~$\infty$; it is given by the Riemann-Hurwitz formula:
\begin{equation}
m = 2g -2 +l(\mu) + s \; .
\end{equation}
\end{definition}

Note that the Hurwitz numbers defined here differ slightly from those in~\cite{Joh10} in that we do not remember the ordering of the branch points over~$\infty$, reflected in the factor~$1/s!$. Write
\begin{equation}
\Fpart_{g,\ell}(p_1, p_2, \ldots) := \sum_{\mu \colon l(\mu) = \ell} \frac{\hpart_{g;\mu}}{m!}\; p_{\mu_1} \cdots p_{\mu_n} 
\end{equation}
for the generating series of genus~$g$, $q$-double Hurwitz numbers whose partition $\mu$ has $\ell$ parts. The full generating series is given by
\begin{align}
&\log\Zpart(p_1, p_2, \ldots; \lambda) :=  
\sum_{g,\ell} \Fpart_{g,\ell}(p_1, p_2, \ldots) 
\lambda^{2g - 2 + \ell}
\\ \notag
&= \sum_{g,\mu} \frac{\hpart_{g;\mu}}{m!}\; \lambda^{2g-2+l(\mu)} p_{\mu_1} \cdots p_{\mu_{l(\mu)}}
\\ \notag
&=  \la \exp\left(
\sum_{i=1}^\infty \frac{\alpha_i p_i}{i\lambda^{i/q}}\right)
 \exp\left([w^2]\tilde{\E}_0(w)\lambda\right)\exp\left(\frac{\alpha_{-q}}{q}\right)\ra^\circ.
\end{align}

\subsection{Spectral curve from $(0,1)$ geometry}
To find an equation for the spectral curve, we compute the $(g,n) = (0,1)$ part of the generating function
\begin{equation}
\Fpart_{0,1} (\mathbf{p})= [w_1^2 \cdots w_{n-1}^2] \sum_{n=1}^\infty p_{nq} \la \frac{\alpha_{nq}}{nq} \cdot \prod_{i=1}^{n-1} \frac{\tilde{\E}_0(w_i)}{(n-1)!} \cdot \frac{\alpha_{-q}^n}{q^n n!}  \ra^\circ \; .
\end{equation} 
Using the commutation relations~\eqref{eq:comm} to commute the operator~$\alpha_{nq}$ to the right, we obtain:
\begin{equation}
\Fpart_{0,1}(\mathbf{p}) = \sum_{n=1}^\infty \frac{(nq)^{n-2}}{n!} p_{nq} . 
\end{equation}

We will abuse notation and write $\Fpart_{0,1}(x) = \Fpart_{0,1}(\mathbf{p})|_{p_i \mapsto x^i}$ for the principal specialization of~$\Fpart_{0,1}$. 

\begin{remark}\label{re:spectralCurve}
Suppose the generating function for these Hurwitz numbers comes from a spectral curve in~$\C^2$. Denote by~$x$ and~$y$ the coordinates on the two copies of~$\C$. Then by the topological recursion theory, the one-form $\omega_{0,1}(x) = \mathrm{d}\Fpart_{0,1}(x)$ should be equal to~$y(x)\mathrm{d}x$. Sometimes, it will be more natural to think of the spectral curve as living in~$\C^* \times \C$ or in~$(\C^*)^2$. In that case $\omega_{0,1}(x)$ should be equal to $y(x)\frac{\mathrm{d}x}{x}$ or~$\log(y) \frac{\mathrm{d}x}{x}$ respectively. 
\end{remark}

We define an auxiliary function. Let~$W$ be the main branch of the Lambert function~\cite{Knuth}. It has a power-series expansion around zero with radius of convergence of~$1/e$ given by
\begin{equation}
W(z) = - \sum_{n=1}^\infty \frac{n^{n-1}}{n!} (-z)^n .
\end{equation}
and has the property that 
\begin{equation}\label{eq:Wequation}
W(z) e^{W(z)} = z .
\end{equation}

Using this definition, we have
\begin{equation}\label{eq:omega01}
\quad \omega_{0,1}(x) = \mathrm{d}\Fpart_{0,1}(x) 
 = \frac{1}{q} \sum_{n=1}^\infty \frac{n^{n-1}}{n!} (qx^q)^n \frac{\mathrm{d}x}{x} 
= -\frac{1}{q} W(-q x^q) \frac{\mathrm{d}x}{x}\ , \end{equation}
where the last equality is true as long as $|x| \leq (qe)^{-1/q}$.

Therefore, Remark~\ref{re:spectralCurve} leads us to think of the spectral curve $\Spart$ as living in $\C^* \times \C$, given by the equation 
\begin{equation}
\Spart\colon \ y = -\frac{1}{q}W(-qx^q) .
\end{equation}
which can be rewritten to get
\begin{equation}
-qx^q = -qy e^{-qy} \quad \Leftrightarrow \quad x = y^{1/q} e^{-y} . 
\end{equation}

\subsection{Principal specialization}
 Here we once again abuse notation and write 
\begin{equation}
\Zpart(x;\lambda) := \Zpart(\mathbf{p};\lambda)|_{p_i \mapsto x^i}  
\end{equation}
for the principal specialization of~$\Zpart$.

Let~$s_{\sigma}(\mathbf{p})$ be the Schur function corresponding to a partition~$\sigma$, which is given as the following vacuum expectation value in the infinite wedge space
\begin{equation}
s_{\sigma}(\mathbf{p}) := \la 0 \left| \ \exp\left(\sum_{i=0}^\infty \frac{\alpha_i p_i}{i}\right)\ \right| v_{\sigma} \ra . 
\end{equation} 
It is a standard fact in the theory of Schur functions that its principal specialization is given by \begin{equation}
s_{\sigma}(\mathbf{p})|_{p_i \mapsto x^i} = 
	\begin{cases}
		x^l &\text{ if } \sigma = (l, 0, \ldots) \text{ for some } l \\
		0 &\text{ otherwise . } 
	\end{cases}
\end{equation}
Using this it is easy to see that the principal specialization of~$\Zpart$ is given by
\begin{equation}
\Zpart(x;\lambda) = \sum_{i=0}^\infty \frac{x^{iq}}{i! (\lambda q)^i} \exp\left(\lambda\frac{ (iq-\frac{1}{2})^2 - (-\frac{1}{2})^2}{2}\right) .
\end{equation}

To find an operator that annihilates this power-series, we proceed as follows. Denote the $i$th summand in~$\Zpart(x;\lambda)$ by~$a_i$:
\begin{equation}
a_i := \frac{x^{iq}}{i! (\lambda q)^i} \exp\left(\lambda\frac{ (iq-\frac{1}{2})^2 - (-\frac{1}{2})^2}{2}\right) .
\end{equation}
Then 
\begin{equation}
\frac{a_{i+1}}{a_i} = \frac{x^q}{(i+1)\lambda q} e^{\lambda\left(iq^2 + \frac{q(q-1)}{2}\right)}, 
\end{equation}
which implies that the coefficients of~$\Zpart(x;\lambda)$ are related by
\begin{equation}
\lambda q(i+1) a_{i+1} = \left( x e^{\lambda \frac{q-1}{2}}\right)^q e^{\lambda i q^2} a_i .
\end{equation}

In terms of operators, this can be rewritten as
\begin{equation}\label{eq:spectralAnnihilate}
\lambda x \frac{\mathrm{d}}{\mathrm{d}x} a_{i+1} - \left( x e^{\lambda \frac{q-1}{2}}\right)^q e^{q \lambda x\frac{\mathrm{d}}{\mathrm{d}x}} a_i = 0,
\end{equation}
which implies that the operator
\begin{equation}
\lambda x \frac{\mathrm{d}}{\mathrm{d}x} - \left( x e^{\lambda \frac{q-1}{2}}\right)^q e^{q \lambda x\frac{\mathrm{d}}{\mathrm{d}x}}
\end{equation}
annihilates~$\Zpart(x;\lambda)$.

\subsection{Quantization} 
We show that the operator that annihilates the principal specialization of~$\Zpart$ can be obtained as a quantization of the equation of the spectral curve~$\Spart$.

The spectral curve $\Spart$ is defined in $\C^* \times \C$, where the symplectic form is $\lambda \mathrm{d}(\log(x)) \wedge \mathrm{d} y$, so we have the following rules of quantization:
\begin{equation} \label{eq:quant}
	\begin{cases}
		&\hat{x} = x\cdot \\
		&\hat{y} = \lambda \frac{\mathrm{d}}{\mathrm{d}(\log(x))}=\lambda x \frac{\mathrm{d}}{\mathrm{d}x}	
	\end{cases} .
\end{equation}

In order to have the right ordering, we rewrite the equation for $\Spart$ as follows:
\begin{equation}\label{eq:spectralAlternative}
\Spart \colon \ y - \left(e^{\frac{q-1}{2} y} x e^{-\frac{q-1}{2} y} \right)^q e^{qy} = 0 . 
\end{equation}

\begin{theorem} Quantization of the equation of $\Spart$ in this form annihilates $\Zpart(x,\lambda)$.
\end{theorem}

\begin{proof} Indeed, direct computation implies that
\begin{equation}
\hat y - \left(e^{\frac{q-1}{2} \hat y} \hat x e^{-\frac{q-1}{2} \hat y} \right)^q e^{q \hat y}
=
\lambda x \frac{\mathrm{d}}{\mathrm{d}x} - \left( x e^{\lambda \frac{q-1}{2}}\right)^q e^{q \lambda x\frac{\mathrm{d}}{\mathrm{d}x}},
\end{equation}
and we have seen in the previous section that this operator annihilates $\Zpart(x,\lambda)$.
\end{proof}

We see that $q$-double Hurwitz numbers are an example of a theory obeying a Schr\"{o}dinger-like equation with respect to the quantization of the spectral curve as expected by~\cite{GuSu}, but contrary to the previous known cases~\cite{MS, Zhou1, Zhou2} we have to take a non-trivial ordering of the operators to obtain this result.

\section{$r$-Spin Hurwitz numbers}\label{sec:rspin}
In this section, we look at the $r$-spin single Hurwitz numbers. They can be defined as vacuum expectation values in the infinite wedge space in the following way.

\begin{definition} We define the (connected) $r$-spin Hurwitz numbers as
\begin{equation}
\hspin_{g,\mu} := \la \prod_{i=1}^{l(\mu)} \frac{\alpha_{\mu_i}}{\mu_i} \cdot \left(r! [w^{r+1}] \tilde{\E}_0(w)\right)^m \cdot  \frac{(\alpha_{-1})^{|\mu|}}{|\mu | !} \ra^\circ ,
\end{equation}
where~$m$ is the number of ramification points other than~$0$, which is given by the Riemann-Hurwitz formula:
\begin{equation}
m = \frac{2g - 2 + l(\mu) + |\mu|}{r} .
\end{equation}
\end{definition}

Note that for~$r=1$, this definition reduces to the definition of ordinary single Hurwitz numbers. Note also that there are different conventions on the coefficient of $[w^{r+1}] \tilde{\E}_0(w)$ in different sources; in particular, a different convention is used in~\cite{ShaSpiZvo}.

Similar to to previous section, we denote by $\Fspin_{g,n}(\mathbf{p})$ the generating function for genus~$g$, $r$-spin Hurwitz numbers $\hspin_{g,\mu}$ whose partition $\mu$ has~$\ell$ parts. That is,
\begin{equation}
\Fspin_{g,\ell}(\mathbf{p}) := \sum_{\mu\colon l(\mu) = \ell} \hspin_{g;\mu} p_{\mu_1} \cdots p_{\mu_\ell} . 
\end{equation}
For the full generating function~$\Zspin$ we then have
\begin{align}\label{eq:Zspin}
& \log \Zspin(\mathbf{p}, \lambda) := \sum_{g,\ell} \Fspin_{g,\ell}(\mathbf{p})\lambda^{2g-2+\ell} 
\\ \notag
& = \la \exp\left(\sum_{i=1}^\infty \frac{\alpha_i p_i}{i\lambda^i}\right) \exp\left(r! [w^{r+1}] \tilde{\E}_0(w) \lambda^r\right) \exp(\alpha_{-1})\ra^\circ .
\end{align} 

\subsection{Spectral curve from $(0,1)$ geometry} To find an equation for the spectral curve, we compute the $(g,n) = (0,1)$ part of the generating function. Commuting the operator~$\alpha_d$ responsible for the total ramification over~$0$ in~\eqref{eq:Zspin} to the right, we obtain
\begin{equation}
\Fspin_{0,1}(\mathbf{p}) = \sum_{n = 0}^\infty \frac{(rn+1)^{n-2}}{n!} p_{rn+1} . 
\end{equation}
Applying the principal specialization, this means that
\begin{equation}
\Fspin_{0,1}(x) = \sum_{n = 0}^\infty \frac{(rn+1)^{n-2}}{n!} x^{rn+1} ,
\end{equation}
which leads to
\begin{equation}\label{eq:omega-1-0-r}
\omega_{0,1}(x) = \mathrm{d}\Fspin_{0,1} = \sum_{n = 0}^\infty \frac{(rn+1)^{n-1}}{n!} x^{rn+1} \frac{\mathrm{d} x}{x} . 
\end{equation}

We use the following formula from~\cite{Knuth}:
\begin{equation}
\left(\frac{W(x)}{x}\right)^{\alpha} = \sum_{n=0}^\infty \frac{\alpha(n+\alpha)^{n-1}}{n!}(-x)^n 
\end{equation} 
to express the right hand side of Equation~\eqref{eq:omega-1-0-r} in a more convenient way. That is
\begin{align}
& \sum_{n = 0}^\infty \frac{(rn+1)^{n-1}}{n!} x^{rn+1} 
 = x \sum_{n=0}^\infty \frac{\frac{1}{r}(n+ \frac{1}{r})^{n-1}}{n!}(rx^r)^n
\\ \notag
&
= x \left( \frac{W(-rx^r)}{-rx^r}\right)^{1/r} = \frac{W(-rx^r)^{1/r}}{(-r)^{1/r}}
\end{align}
Thus, by Remark~\ref{re:spectralCurve} we arrive at the following equation for the spectral curve~$\Sspin$ in $\C^* \times \C$:
\begin{equation} \label{eq:SpectralCurve2}
\Sspin \colon \ y = \left(\frac{W(-rx^r)}{-r}\right)^{1/r} \quad \Leftrightarrow \quad x= y e^{-y^r} .
\end{equation}

\subsection{Principal specialization}
Once again we look at the principal specialization of the full generating function
\begin{align}
\quad \quad \Zspin(x;\lambda) & = \Zspin(x;\lambda)|_{p_i \mapsto x^i} \\ \notag
& = \sum_{d=0}\frac{x^d}{\lambda^d d!} \exp\left( \lambda^r \frac{(d-\frac{1}{2})^{r+1} - (-\frac{1}{2})^{r+1}}{r+1}\right).
\end{align}

We define~$a_d$ to be the $d$th summand this expression. The quotient of $a_{d+1}$ and $a_d$ is given by
\begin{equation}
\frac{a_{d+1}}{a_d} = \frac{x}{\lambda(d+1)} \exp\left(\lambda^r\frac{(d+\frac{1}{2})^{r+1} - (d-\frac{1}{2})^{r+1}}{r+1} \right),
\end{equation}
which is equivalent to
\begin{equation}\label{eq:almostQuantized}
(d+1)\lambda a_{d+1} = x \exp\left(\lambda^r\frac{(d+\frac{1}{2})^{r+1} - (d-\frac{1}{2})^{r+1}}{r+1} \right) a_d .
\end{equation}

To get this into a more convenient form to compare later on with quantization, we define an operator
\begin{equation}
\mathcal{A} := x^{\frac{3}{2}} \exp\left(\frac{x^{-1}\sum_{i=0}^r\left(\lambda x \xdiff\right)^i x \left(\lambda x \xdiff\right)^{r-i}}{r+1}\right)x^{-\frac{1}{2}} .
\end{equation}
Observe that
\begin{align}
\mathcal{A} x^n & = \exp\left(\frac{\lambda^r}{r+1} \sum_{i=0}^r (n+\frac{1}{2})^i (n-\frac{1}{2})^{r-i}\right) x^{n+1}\\ \notag
& = \exp\left(\frac{\lambda^r}{r+1} \left((n+\frac{1}{2})^{r+1} - (n-\frac{1}{2})^{r+1}\right)\right) x^{n+1} .
\end{align}
Thus, equation~\eqref{eq:almostQuantized} implies that
\begin{equation}
(\lambda x \xdiff - \mathcal{A}) \Zspin(x;\lambda) = 0. 
\end{equation}

\subsection{Quantization} 
We show that the operator that annihilates the principal specialization of~$\Zspin$ can be obtained as a quantization of the equation of the spectral curve~$\Sspin$.

We can rewrite the equation of the spectral curve~\eqref{eq:SpectralCurve2} as
\begin{equation}\label{eq:sspin}
\Sspin \colon \ y - x^{\frac{3}{2}} \exp\left(\frac{\sum_{i=0}^r x^{-1} y^i x y^{r-i}}{r+1}\right) x^{-\frac{1}{2}} = 0 . 
\end{equation}

\begin{theorem} Quantization of the equation of $\Sspin$ in this form annihilates $\Zspin(x,\lambda)$.
\end{theorem}

\begin{proof} Indeed, apllying the standard quantization~\eqref{eq:quant} to Equation~\eqref{eq:sspin} we obtain the operator $\lambda x \xdiff - \mathcal{A}$.
\end{proof}

\section{Mixed case}
In this section we provide a slight generalization of the previous two sections, were we look at (connected) $r$-spin 
$q$-double Hurwitz numbers~$\hmix_{g,\mu}$. Since the computations are basically the same as in the previous two sections, we just give the main formulas. Note that for $r = 1$ this reduces to the computations of Section~\ref{sec:rpart}, and for $q = 1$ this reduces to those of Section~\ref{sec:rspin}.

These Hurwitz numbers are given as vacuum expectation values by:
\begin{equation}
\hmix_{g,\mu} =
\la
\prod_{i=1}^\ell \frac{\alpha_{\mu_i}}{\mu_i}
\cdot
\left(
r! [z^{r+1}] \tilde\E_0(z)
\right)^m
\cdot
\frac{\left(
\alpha_{-q}
\right)^s}{q^s s!} 
\ra^\circ \: .
\end{equation}
Here the degree of the covering is given by $d=\sum_{i=1}^{l(\mu)} \mu_i = qs$, and the Riemann-Hurwitz formula reads $2g-2+l(\mu)= mr-s$. 

The full generating function is given by
\begin{align}
& \log \Zmix(\mathbf{p};\lambda) = \sum_{g,\mu} \frac{\hmix_{g;\mu}}{m!} \lambda^{2g-2+l(\mu)} p_{\mu_1} \cdots p_{\mu_{l(\mu)}} \\ \notag
& = \la \exp\left(\sum_{i=1}^n \frac{\alpha_i p_i}{i \lambda^{i/q}}\right)\exp\left(r! [w^{r+1}] \tilde{\E}_0(w)\lambda^r \right) \exp\left(\frac{\alpha_{-q}}{q}\right) \ra^\circ \; ,
\end{align}
and the $(0,1)$-function is given by
\begin{equation}
F_{0,1}(x) = q \sum_{n=0}^\infty \frac{((nr+1)q)^{n-2}}{n!} x^{(nr+1)q}
\end{equation}

This leads to the following spectral curve:
\begin{equation}\label{eq:Smix}
S\colon \ x = y^{1/q} e^{-y^{r}},
\end{equation}
which means that
\begin{equation}
y=-\left(\frac{1}{rq}\right)^\frac{1}{r}W(-rqx^{rq})^{\frac{1}{r}},
\end{equation}
where $W$ is the standard Lambert function.

The principal specialization ($p_i\mapsto x^i$) of~$\Zmix$ is given by
\begin{equation}
\Zmix(x,\lambda) = \sum_{n=0}^\infty \frac{x^{qn}}{\lambda^nq^nn!}e^{\frac{\lambda^r}{r+1}\left(
(qn-\frac{1}{2})^{r+1}-(-\frac{1}{2})^{r+1}
\right)} ,
\end{equation}
which is annihilated by the operator
\begin{equation}
\lambda x\frac{d}{dx}-x^{q+1/2} e^{\frac{q}{r+1}\sum_{i=0}^r x^{-q} \left(\lambda x\frac{d}{dx}\right)^i x^{q} \left(\lambda x\frac{d}{dx}\right)^{r-i}} x^{-1/2}
\end{equation}

This operator dequantizes to $y- x^q\exp(qy^r)$, which is equivalent to the equation~\eqref{eq:Smix} of the spectral curve~$\Smix$ computed from the $(0,1)$-geometry. 

Furthermore, one sees immediately that under the specializations $(r,q)=(1,q)$ and $(r,1)$ we recover all the formulas we had in Sections~\ref{sec:rpart} and~\ref{sec:rspin}.

\end{document}